\documentclass[a4paper,oneside,10pt]{article}%
\usepackage{amsmath}
\usepackage{amsfonts}
\usepackage{amssymb}
\usepackage{graphicx}
\usepackage{color}
\usepackage[square,numbers,sort&compress]{natbib}%
\providecommand{\U}[1]{\protect \rule{.1in}{.1in}}

\newtheorem{theorem}{Theorem}[section]

\newtheorem{corollary}[theorem]{Corollary}

\newtheorem{definition}[theorem]{Definition}

\newtheorem{example}[theorem]{Example}

\newtheorem{lemma}[theorem]{Lemma}

\newtheorem{proposition}[theorem]{Proposition}

\newenvironment{proof}[1][Proof]{\noindent \textbf{#1.} }{\  $\Box$}
\numberwithin{equation}{section}

\begin{document}

\title{Necessary and sufficient conditions for optimal control of coupled
forward-backward stochastic control systems driven by $G$-Brownian motion}
\author{Xiaojuan Li\thanks{Department of Mathematics, Qilu Normal University, Jinan
250200, China. lxj110055@126.com. }
\and Mingshang Hu \thanks{Zhongtai Securities Institute for Financial Studies,
Shandong University, Jinan, Shandong 250100, China. humingshang@sdu.edu.cn.
Research supported by the Taishan Scholars Project (NO. tstp202607025) and the National Natural Science Foundation of China (NO. 12326603).} }
\maketitle

\textbf{Abstract}. In this paper, we study a coupled forward-backward
stochastic control system driven by $G$-Brownian motion. Based on some
detailed estimates for control system, we obtain the variational inequality
for cost functional, which is a necessary condition for optimality.
Furthermore, we prove that this necessary condition is also a sufficient
condition under the usual convexity assumptions.



\addcontentsline{toc}{section}{\hspace*{1.8em}Abstract}

\section{Introduction}

In order to study the issue of volatility uncertainty in finance, Peng
\cite{P07a, P08a} introduced the theory of $G$-expectation, and established
the theory of stochastic differential equation driven by $G$-Brownian motion
($G$-SDE). If $G$ is a linear function, then $G$-Brownian motion is classical
Brownian motion. Furthermore, Hu et al. \cite{HJPS1, HJPS} established the
theory of backward stochastic differential equation driven by $G$-Brownian
motion ($G$-BSDE). Another important fully nonlinear BSDE ($2$BSDE) was
established by Soner et al. \cite{STZ11}. The formulations and methods for
$G$-BSDE and $2$BSDE are different, and the latest developments can be found
in \cite{DK, FPW, HJL1, HYH, HTW, LRT, MPZ} and the references therein.

It is well known that the maximum principle is an important tool for dealing
with stochastic optimal control problems. It is a necessary condition for
optimal control to be satisfied and also a sufficient condition under the
usual convexity assumptions (see \cite{YZ}). Xu \cite{Xu} first studied the
following stochastic control system under $G$-expectation framework (for
simplicity, the state process $X$ and $G$-Brownian motion $B$ are
$1$-dimensional in this section):%
\begin{equation}
dX_{t}^{u}=b(t,X_{t}^{u},u_{t})dt+\sigma(t,X_{t}^{u},u_{t})dB_{t},\text{
}X_{0}^{u}=x_{0}\in \mathbb{R}, \label{e0-1}%
\end{equation}
where the control domain $U$ is a given nonempty convex set of $\mathbb{R}%
^{m}$, and the set of admissible controls is denoted by $\mathcal{U}[0,T]$ in
this section. The cost functional is defined as follows:%
\begin{equation}
J(u):=\mathbb{\hat{E}}\left[  \phi(X_{T}^{u})-\int_{0}^{T}f(t,X_{t}^{u}%
,u_{t})dt\right]  . \label{e0-4}%
\end{equation}
The optimal control problem is to minimize $J(u)$ over $\mathcal{U}[0,T]$. Xu
obtained the variational equation for (\ref{e0-1}), and the necessary
condition for optimal control to be satisfied by using the subadditivity of
$G$-expectation. However, this necessary condition is not a sufficient one
under the usual convexity assumptions. Furthermore, Hu and Ji \cite{HJ1}
studied the following forward-backward stochastic control system:%
\begin{equation}
\left \{
\begin{array}
[c]{rl}%
dX_{t}^{u}= & b(t,X_{t}^{u},u_{t})dt+\sigma(t,X_{t}^{u},u_{t})dB_{t},\\
dY_{t}^{u}= & f(t,X_{t}^{u},Y_{t}^{u},Z_{t}^{u},u_{t})dt+Z_{t}^{u}%
dB_{t}+dK_{t}^{u},\\
X_{0}^{u}= & x_{0}\in \mathbb{R},\text{ }Y_{T}^{u}=\phi(X_{T}^{u}),\text{ }%
t\in \lbrack0,T].
\end{array}
\right.  \label{e0-2}%
\end{equation}
The cost functional is defined by $J(u):=Y_{0}^{u}$, and the optimal control
problem is to minimize $J(u)$ over $\mathcal{U}[0,T]$. If $f(\cdot)$ does not
contain $(Y^{u},Z^{u})$ in (\ref{e0-2}), then $Y_{0}^{u}$ is equal to the
right side of (\ref{e0-4}). By introducing the linearization and weak
convergence methods, Hu and Ji obtained the derivative for $J(u)$ and the
maximum principle, which is a sufficient condition under the usual convexity
assumptions. It is worth noting that the variational equation for
$(Y^{u},Z^{u},K^{u})$ generally does not exist, which is completely different
from the classical probabilistic case.

In this paper, we consider the following coupled forward-backward stochastic
control system:
\begin{equation}
\left \{
\begin{array}
[c]{rl}%
dX_{t}^{u}= & b(t,X_{t}^{u},Y_{t}^{u},u_{t})dt+\sigma(t,X_{t}^{u},Y_{t}%
^{u},u_{t})dB_{t},\\
dY_{t}^{u}= & f(t,X_{t}^{u},Y_{t}^{u},Z_{t}^{u},u_{t})dt+Z_{t}^{u}%
dB_{t}+dK_{t}^{u},\\
X_{0}^{u}= & x_{0}\in \mathbb{R},\text{ }Y_{T}^{u}=\phi(X_{T}^{u}),\text{ }%
t\in \lbrack0,T].
\end{array}
\right.  \label{e0-3}%
\end{equation}
The cost functional is defined by $J(u):=Y_{0}^{u}$, and the optimal control
problem is to minimize $J(u)$ over $\mathcal{U}[0,T]$. When $G$ is linear, the
well-posedness of (\ref{e0-3}) can be found in \cite{MY} and the references
therein, and the corresponding maximum principle can be found in \cite{HJX,
PW, SW} and the references therein. For the nonlinear $G$, Lu and Song
\cite{LS}, and Zheng \cite{Zh} obtained that (\ref{e0-3}) has a unique $L^{2}%
$-solution for small $T$, Wang and Yuan \cite{WY} obtained that (\ref{e0-3})
has a minimal solution under the assumptions that $b(\cdot)$ is increasing in
$y$ and $\sigma(\cdot)$ does not contain $y$, and Li \cite{Li} obtained that
(\ref{e0-3}) has a unique $L^{p}$-solution with $p\geq2$ for arbitrary $T$
under weakly coupling condition. Here, we study the control system
(\ref{e0-3}) under the weakly coupling condition in \cite{Li}. In particular,
the $L^{p}$-estimates (see Theorem \ref{th1-1}) are tools for deriving the
necessary conditions for optimality.

Since $G$-SDE in \ref{e0-3} contains $Y^{u}$ and $Y^{u}$ does not have the
variational equation, $X^{u}$ also does not have the variational equation.
This point is completely different from the control system (\ref{e0-2}), which
leads to the linearization and weak convergence methods no longer being
applicable. To overcome this difficulty, we obtain detailed estimates for
$X^{\varepsilon}-X^{\ast}$, $Y^{\varepsilon}-Y^{\ast}$ and $Z^{\varepsilon
}-Z^{\ast}$ (see Lemmas \ref{le2-1} and \ref{le2-2}) by using $L^{p}%
$-estimates, and particularly, we also get better estimates for $Z^{\varepsilon}-Z^{\ast}$
under $\mathcal{P}^{\ast}$ and $\mathcal{P}_{\varepsilon}^{\ast}$ (see
(\ref{e2-4})). Based on these estimates, we obtain the variational inequality
for cost functional, which is a necessary condition for optimality.
Furthermore, we prove that this necessary condition is also a sufficient
condition under the usual convexity assumptions.

This paper is organized as follows. In Section 2, we recall some basic results
of $G$-expectation and give the formulation of our control problem. In Section
3, we obtain the variational inequality for cost functional. In Section 4, we
give the sufficient conditions for optimality.

\section{Preliminaries and formulation of the control problem}

For the sake of simplicity, we only consider $1$-dimensional $G$-Brownian
motion, and the multi-dimensional cases are treated in a similar way.

Let $T>0$ be given and let $\Omega_{T}=C_{0}[0,T]$ be the space of
$\mathbb{R}$-valued continuous functions on $[0,T]$ with $\omega_{0}=0$. The
canonical process $B_{t}(\omega):=\omega_{t}$, for $\omega \in \Omega_{T}$ and
$t\in \lbrack0,T]$. For each given $t\in \lbrack0,T]$, set%
\[
Lip(\Omega_{t}):=\{ \varphi(B_{t_{1}},B_{t_{2}}-B_{t_{1}},\ldots,B_{t_{N}%
}-B_{t_{N-1}}):N\geq1,t_{1}<\cdots<t_{N}\leq t,\varphi \in C_{b.Lip}%
(\mathbb{R}^{N})\},
\]
where $C_{b.Lip}(\mathbb{R}^{N})$ denotes the space of bounded Lipschitz
functions on $\mathbb{R}^{N}$.

For each given $0<\underline{\sigma}<\bar{\sigma}<\infty$, define%
\[
G(a)=\frac{1}{2}(\bar{\sigma}^{2}a^{+}-\underline{\sigma}^{2}a^{-})\text{ for
}a\in \mathbb{R}\text{.}%
\]
Peng \cite{P07a, P08a} constructed the $G$-expectation space $(\Omega
_{T},Lip(\Omega_{T}),\mathbb{\hat{E}},(\mathbb{\hat{E}}_{t})_{t\in \lbrack
0,T]})$ as follows: for each $0<t_{1}<\cdots<t_{N}=T$ and $\xi=\varphi
_{N}(B_{t_{1}},B_{t_{2}}-B_{t_{1}},\ldots,B_{t_{N}}-B_{t_{N-1}})\in
Lip(\Omega_{T})$, define%
\[
\mathbb{\hat{E}}_{t_{i}}[\xi]:=\varphi_{i}(B_{t_{1}},\ldots,B_{t_{i}%
}-B_{t_{i-1}})\text{ for }i=N,\ldots,0,
\]
where $t_{0}=0$, $\mathbb{\hat{E}}[\xi]:=\mathbb{\hat{E}}_{t_{0}}[\xi]$,
$\varphi_{i}(x_{1},\ldots,x_{i})=u^{\varphi_{i+1}(x_{1},\ldots,x_{i},\cdot
)}(t_{i+1}-t_{i},0)$ and $u^{\varphi_{i+1}(x_{1},\ldots,x_{i},\cdot)}$ is the
viscosity solution (see \cite{CIP}) of the following PDE:
\[
\left \{
\begin{array}
[c]{l}%
\partial_{t}u-G(\partial_{xx}^{2}u)=0,\\
u(0,x)=\varphi_{i+1}(x_{1},\ldots,x_{i},x),
\end{array}
\right.
\]
for $i=N-1,\ldots,0$. The canonical process $B$ is called $G$-Brownian motion
under $\mathbb{\hat{E}}$.

For each $t\in \lbrack0,T]$ and $p\geq1$, denote by $L_{G}^{p}(\Omega_{t})$ the
completion of $Lip(\Omega_{t})$ under the norm $||\xi||_{L_{G}^{p}%
}:=(\mathbb{\hat{E}}[|\xi|^{p}])^{1/p}$. $\mathbb{\hat{E}}$ and $\mathbb{\hat
{E}}_{t}$ can be continuously extended to $L_{G}^{1}(\Omega_{T})$ under the
norm $||\cdot||_{L_{G}^{1}}$.

The following representation theorem for $G$-expectation can be found in
\cite{DHP11, HP09}.

\begin{theorem}
\label{th-re-1}There exists a unique weakly compact and convex set of
probability measures $\mathcal{P}$ on $(\Omega_{T},\mathcal{F}_{T})$ such that%
\[
\mathbb{\hat{E}}[\xi]=\sup_{P\in \mathcal{P}}E_{P}[\xi]\text{ for all }\xi \in
L_{G}^{1}(\Omega_{T}),
\]
where $\mathcal{F}_{t}=\mathcal{B}(\Omega_{t})=\sigma(\{B_{s}:s\leq t\})$ for
$t\leq T$.
\end{theorem}

The capacity associated with $\mathcal{P}$ is defined by%
\[
c(A):=\sup_{P\in \mathcal{P}}P(A)\text{ for }A\in \mathcal{F}_{T}.
\]
A set $A\in \mathcal{F}_{T}$ is polar if $c(A)=0$. A property holds
\textquotedblleft quasi-surely" (q.s. for short) if it holds outside a polar
set, and we do not distinguish two random variables $X$ and $Y$ if $X=Y$ q.s.

The following Doob's inequality for $G$-martingale can be found in \cite{STZ,
Song11}.

\begin{theorem}
\label{Doob}Let $\xi \in L_{G}^{p}(\Omega_{T})$ with $p>1$. Then%
\[
\mathbb{\hat{E}}\left[  \sup_{t\leq T}\mathbb{\hat{E}}_{t}[\xi]\right]
\leq \frac{p}{p-1}\left(  \mathbb{\hat{E}}[|\xi|^{p}]\right)  ^{1/p}.
\]

\end{theorem}

Let $M_{G}^{0}(0,T)$ be the space of the following simple processes: for each
integer $N\geq1$ and $0=t_{0}<\cdots<t_{N}=T$,%
\[
\eta_{t}=\sum_{i=0}^{N-1}\xi_{i}I_{[t_{i},t_{i+1})}(t),\text{ }t\in
\lbrack0,T],
\]
where $\xi_{i}\in Lip(\Omega_{t_{i}})$ for $i=0,1,\ldots,N-1$. Let $S_{G}%
^{0}(0,T)$ be the space of the following continuous processes: for each
integer $N\geq1$ and $0<t_{1}<\cdots<t_{N}=T$,%
\[
\eta_{t}=\phi(t,B_{t_{1}\wedge t},\ldots,B_{t_{N}\wedge t}),\text{ }%
t\in \lbrack0,T],
\]
where $\phi \in C_{b.Lip}(\mathbb{R}^{1+N})$.

For each given $p$, $\bar{p}\geq1$, let $M_{G}^{p,\bar{p}}(0,T)$ denote the
completion of $M_{G}^{0}(0,T)$ under the norm $||\eta||_{M_{G}^{p,\bar{p}}%
}:=\left(  \mathbb{\hat{E}}\left[  \left(  \int_{0}^{T}|\eta_{t}%
|^{p}dt\right)  ^{\bar{p}/p}\right]  \right)  ^{1/\bar{p}}$ and let $S_{G}%
^{p}(0,T)$ denote the completion of $S_{G}^{0}(0,T)$ under the norm
$||\eta||_{S_{G}^{p}}:=\left(  \mathbb{\hat{E}}[\sup_{t\leq T}|\eta_{t}%
|^{p}]\right)  ^{1/p}$. Set $M_{G}^{p}(0,T)=M_{G}^{p,p}(0,T)$. For each
$\eta \in M_{G}^{2,\bar{p}}(0,T)$, the $G$-It\^{o} integral $\int_{0}^{T}%
\eta_{t}dB_{t}$ is well defined (see \cite{P2019} for more details).

In the following part of this section, we give the formulation of our control
problem. Let $\beta>2$ be a given constant throughout the paper, and let the
control domain $U$ be a given nonempty convex set of $\mathbb{R}^{m}$. The set
of admissible controls is denoted by%
\[
\mathcal{U}[0,T]=M_{G}^{2,\beta}(0,T;U):=\{u\in M_{G}^{2,\beta}(0,T;\mathbb{R}%
^{m}):u_{t}\in U\text{ for }t\in \lbrack0,T]\},
\]
where%
\[
M_{G}^{2,\beta}(0,T;\mathbb{R}^{m})=\{(u_{1},\ldots,u_{m})^{T}:u_{i}\in
M_{G}^{2,\beta}(0,T)\text{ for }i=1,\ldots,m\}.
\]
For simplicity, we consider the followng $1$-dimensional coupled
forward-backward stochastic control system. The results are similar in the
general multi-dimensional case studied in \cite{Li}. For each $u\in
\mathcal{U}[0,T]$,%

\begin{equation}
\left \{
\begin{array}
[c]{rl}%
dX_{t}^{u}= & b(t,X_{t}^{u},Y_{t}^{u},u_{t})dt+\sigma(t,X_{t}^{u},Y_{t}%
^{u},u_{t})dB_{t},\\
dY_{t}^{u}= & f(t,X_{t}^{u},Y_{t}^{u},Z_{t}^{u},u_{t})dt+Z_{t}^{u}%
dB_{t}+dK_{t}^{u},\\
X_{0}^{u}= & x_{0}\in \mathbb{R},\text{ }Y_{T}^{u}=\phi(X_{T}^{u}),\text{ }%
t\in \lbrack0,T],
\end{array}
\right.  \label{e1-1}%
\end{equation}
where $b$, $\sigma:[0,T]\times \Omega_{T}\times \mathbb{R}^{2+m}\rightarrow
\mathbb{R}$, $f:[0,T]\times \Omega_{T}\times \mathbb{R}^{3+m}\rightarrow
\mathbb{R}$ and $\phi:\Omega_{T}\times \mathbb{R}\rightarrow \mathbb{R}$ satisfy
the following assumptions:

\begin{description}
\item[(H1)] For any $(x,y,z,v)\in \mathbb{R}^{3+m}$, $(b(t,x,y,v))_{t\leq T}$,
$(f(t,x,y,z,v))_{t\leq T}\in M_{G}^{1,\beta}(0,T)$, $(\sigma(t,x,y,v))_{t\leq
T}\in M_{G}^{2,\beta}(0,T)$, $\phi(x)\in L_{G}^{\beta}(\Omega_{T})$ with
$\beta>2$.

\item[(H2)] $b$, $\sigma$, $f$ and $\phi$ are differentiable in $(x,y,z,v)$
such that%
\[
L_{1}:=\max \left \{  ||b_{x}||_{\infty}\text{, }||\sigma_{x}||_{\infty}\text{,
}||f_{y}||_{\infty}\text{, }||f_{z}||_{\infty}\right \}  <\infty,\text{ }%
L_{2}:=\max \left \{  ||b_{y}||_{\infty}\text{, }||\sigma_{y}||_{\infty
}\right \}  <\infty,
\]%
\[
L_{3}:=\max \left \{  ||f_{x}||_{\infty}\text{, }||\phi_{x}||_{\infty}\right \}
<\infty,\text{ }L_{4}:=\max \left \{  ||b_{v}||_{\infty}\text{, }||\sigma
_{v}||_{\infty}\text{, }||f_{v}||_{\infty}\right \}  <\infty,
\]
where $||b_{x}||_{\infty}:=\sup \{|b_{x}(t,\omega,x,y,v)|:t\in \lbrack0,T],$
$\omega \in \Omega_{T}$, $x$, $y\in \mathbb{R}$, $v\in U\}$, and other
definitions are similar.

\item[(H3)] Weakly coupling condition: $L_{2}L_{3}\leq \delta$, where the
constant $\delta$ satisfies%
\[
8\delta^{2}T(1+T)[T+4\bar{\sigma}^{2}+2\bar{\sigma}^{2}T(1+2L_{1}+2\bar
{\sigma}^{2}L_{1}^{2})]\exp(T+2L_{1}T+(32\bar{\sigma}^{2}+8\underline{\sigma
}^{-2})L_{1}^{2}T+8L_{1}^{2}T^{2})<1.
\]

\item[(H4)] There exists a modulus of continuity $\rho:[0,\infty
)\rightarrow \lbrack0,\infty)$ such that for each $t\leq T$, $x$, $x^{\prime}$,
$y$, $y^{\prime}$, $z$, $z^{\prime}\in \mathbb{R}$, $v$, $v^{\prime}\in U$,%
\[
|\psi(t,x,y,z,v)-\psi(t,x^{\prime},y^{\prime},z^{\prime},v^{\prime})|\leq
\rho(|x-x^{\prime}|+|y-y^{\prime}|+|z-z^{\prime}|+|v-v^{\prime}|),
\]
where $\psi$ is the derivative of $b$, $\sigma$, $f$, $\phi$ in $(x,y,z,v)$.
\end{description}

\begin{definition}
For each given $p\in \lbrack2,\beta)$ and $u\in \mathcal{U}[0,T]$, $(X^{u}%
$,$Y^{u}$,$Z^{u}$,$K^{u})$ is called an $L^{p}$-solution of (\ref{e1-1}) if

\begin{description}
\item[(i)] $X^{u}$, $Y^{u}\in S_{G}^{p}(0,T)$, $Z^{u}\in M_{G}^{2,p}(0,T)$,
$K_{T}^{u}\in L_{G}^{p}(\Omega_{T})$;

\item[(ii)] $K^{u}$ is a non-increasing process satisfying $K_{0}^{u}=0$ and
$K_{t}^{u}=\mathbb{\hat{E}}_{t}[K_{T}^{u}]$ for $t\in \lbrack0,T]$;

\item[(iii)] $(X^{u}$,$Y^{u}$,$Z^{u}$,$K^{u})$ satisfies (\ref{e1-1}).
\end{description}
\end{definition}

The following theorem can be found in \cite{Li}.

\begin{theorem}
\label{th1-1}Suppose assumptions (H1)-(H3) hold. Then there exists a $\bar
{p}\in(2,\beta)$ such that the following results hold.

\begin{description}
\item[(i)] For each $u\in \mathcal{U}[0,T]$ and $p\in \lbrack2,\bar{p}]$,
(\ref{e1-1}) has a unique $L^{p}$-solution $(X^{u}$,$Y^{u}$,$Z^{u}$,$K^{u})$.

\item[(ii)] For each $u$, $\tilde{u}\in \mathcal{U}[0,T]$ and $p\in
\lbrack2,\bar{p}]$, there exists a constant $C_{1}$ depending on $p$, $T$,
$\bar{\sigma}$, $\underline{\sigma}$, $L_{1}$, $L_{2}$ and $L_{3}$ such that%
\[
\mathbb{\hat{E}}\left[  \sup_{t\leq T}|\hat{X}_{t}|^{p}\right]  \leq
C_{1}\mathbb{\hat{E}}\left[  \left(  \int_{0}^{T}(|\hat{b}_{t}|+|\hat{f}%
_{t}|)dt\right)  ^{p}+\left(  \int_{0}^{T}|\hat{\sigma}_{t}|^{2}dt\right)
^{p/2}\right]  ,
\]
where $\hat{X}_{t}=X_{t}^{\tilde{u}}-X_{t}^{u}$, $\hat{b}_{t}=b(t,X_{t}%
^{u},Y_{t}^{u},\tilde{u}_{t})-b(t,X_{t}^{u},Y_{t}^{u},u_{t})$, $\hat{\sigma
}_{t}=\sigma(t,X_{t}^{u},Y_{t}^{u},\tilde{u}_{t})-\sigma(t,X_{t}^{u},Y_{t}%
^{u},u_{t})$ and $\hat{f}_{t}=f(t,X_{t}^{u},Y_{t}^{u},Z_{t}^{u},\tilde{u}%
_{t})-f(t,X_{t}^{u},Y_{t}^{u},Z_{t}^{u},u_{t})$.

\item[(iii)] For each $u$, $\tilde{u}\in \mathcal{U}[0,T]$ and $p\in
\lbrack2,\bar{p}]$, there exists a constant $C_{2}$ depending on $p$, $T$,
$\bar{\sigma}$, $\underline{\sigma}$ and $L_{1}$ such that for any
$t\in \lbrack0,T]$,%
\[
|Y_{t}^{\tilde{u}}-Y_{t}^{u}|^{p}\leq C_{2}\mathbb{\hat{E}}_{t}\left[  \left(
|\phi(X_{T}^{\tilde{u}})-\phi(X_{T}^{u})|+\int_{0}^{T}|f(s,X_{s}^{\tilde{u}%
},Y_{s}^{u},Z_{s}^{u},\tilde{u}_{s})-f(s,X_{s}^{u},Y_{s}^{u},Z_{s}^{u}%
,u_{s})|ds\right)  ^{p}\right]  .
\]

\end{description}
\end{theorem}

The cost functional of our control problem is defined by
\begin{equation}
J(u):=Y_{0}^{u}, \label{e1-2}%
\end{equation}
where $(X^{u}$,$Y^{u}$,$Z^{u}$,$K^{u})$ is the $L^{2}$-solution of
(\ref{e1-1}). The stochastic optimal control problem is to minimize $J(u)$
over $\mathcal{U}[0,T]$, and a control $u^{\ast}\in \mathcal{U}[0,T]$ is called
optimal if%
\[
J(u^{\ast})=\min_{u\in \mathcal{U}[0,T]}J(u).
\]

\section{Necessary conditions for optimality}

Let $u^{\ast}\in \mathcal{U}[0,T]$ be an optimal control for (\ref{e1-2}) and
let $(X^{\ast}$,$Y^{\ast}$,$Z^{\ast}$,$K^{\ast})$ be the solution of
(\ref{e1-1}) corresponding to $u^{\ast}$. For each fixed $u\in \mathcal{U}%
[0,T]$, set%
\[
u^{\varepsilon}=(1-\varepsilon)u^{\ast}+\varepsilon u=u^{\ast}+\varepsilon
(u-u^{\ast})
\]
for each $\varepsilon \in \lbrack0,1]$. Since the control domain $U$ is a convex
set, we can obtain $u^{\varepsilon}\in \mathcal{U}[0,T]$. Let $(X^{\varepsilon
}$,$Y^{\varepsilon}$,$Z^{\varepsilon}$,$K^{\varepsilon})$ be the solution of
(\ref{e1-1}) corresponding to $u^{\varepsilon}$. Now, we first establish some
relevant estimates, and for the sake of convenience in presentation, the
constant $C$ will change from line to line in the following proof.

\begin{lemma}
\label{le2-1}Suppose assumptions (H1)-(H3) hold. Let $u^{\ast}\in
\mathcal{U}[0,T]$ be an optimal control and let $u\in \mathcal{U}[0,T]$ be
given. Then there exists a constant $C$ depending on $\bar{p}$, $T$,
$\bar{\sigma}$, $\underline{\sigma}$, $L_{1}$, $L_{2}$, $L_{3}$, $L_{4}$, $u$
and $u^{\ast}$ such that for any $p\in \lbrack2,\bar{p}]$,%
\begin{equation}
\mathbb{\hat{E}}\left[  \sup_{t\leq T}|X_{t}^{\varepsilon}-X_{t}^{\ast}%
|^{p}\right]  \leq C\varepsilon^{p}, \label{e2-2}%
\end{equation}
and for any $p^{\prime}\in \lbrack2,(2+\bar{p})/2],$
\begin{equation}
\mathbb{\hat{E}}\left[  \sup_{t\leq T}|Y_{t}^{\varepsilon}-Y_{t}^{\ast
}|^{p^{\prime}}\right]  \leq C\varepsilon^{p^{\prime}}, \label{e2-3}%
\end{equation}
where $\bar{p}\in(2,\beta)$ is defined in Theorem \ref{th1-1}.
\end{lemma}

\begin{proof}
By (ii) of Theorem \ref{th1-1}, we obtain that
\[
\mathbb{\hat{E}}\left[  \sup_{t\leq T}|X_{t}^{\varepsilon}-X_{t}^{\ast}%
|^{\bar{p}}\right]  \leq C\mathbb{\hat{E}}\left[  \left(  \int_{0}^{T}%
(|\hat{b}_{t}|+|\hat{f}_{t}|)dt\right)  ^{\bar{p}}+\left(  \int_{0}^{T}%
|\hat{\sigma}_{t}|^{2}dt\right)  ^{\bar{p}/2}\right]  ,
\]
where the constant $C$ depends on $\bar{p}$, $T$, $\bar{\sigma}$,
$\underline{\sigma}$, $L_{1}$, $L_{2}$ and $L_{3}$, $\hat{b}_{t}%
=b(t,X_{t}^{\ast},Y_{t}^{\ast},u_{t}^{\varepsilon})-b(t,X_{t}^{\ast}%
,Y_{t}^{\ast},u_{t}^{\ast})$, $\hat{\sigma}_{t}=\sigma(t,X_{t}^{\ast}%
,Y_{t}^{\ast},u_{t}^{\varepsilon})-\sigma(t,X_{t}^{\ast},Y_{t}^{\ast}%
,u_{t}^{\ast})$ and $\hat{f}_{t}=f(t,X_{t}^{\ast},Y_{t}^{\ast},Z_{t}^{\ast
},u_{t}^{\varepsilon})-f(t,X_{t}^{\ast},Y_{t}^{\ast},Z_{t}^{\ast},u_{t}^{\ast
})$. It follows from the assumption (H2) that%
\[
|\hat{b}_{t}|\leq L_{4}\varepsilon|u_{t}-u_{t}^{\ast}|\text{, }|\hat{f}%
_{t}|\leq L_{4}\varepsilon|u_{t}-u_{t}^{\ast}|\text{ and }|\hat{\sigma}%
_{t}|\leq L_{4}\varepsilon|u_{t}-u_{t}^{\ast}|.
\]
Then we get%
\begin{align*}
\mathbb{\hat{E}}\left[  \sup_{t\leq T}|X_{t}^{\varepsilon}-X_{t}^{\ast}%
|^{\bar{p}}\right]   &  \leq C\varepsilon^{\bar{p}}\mathbb{\hat{E}}\left[
\left(  \int_{0}^{T}|u_{t}-u_{t}^{\ast}|dt\right)  ^{\bar{p}}+\left(  \int
_{0}^{T}|u_{t}-u_{t}^{\ast}|^{2}dt\right)  ^{\bar{p}/2}\right] \\
&  \leq C\varepsilon^{\bar{p}}\mathbb{\hat{E}}\left[  \left(  \int_{0}%
^{T}|u_{t}-u_{t}^{\ast}|^{2}dt\right)  ^{\bar{p}/2}\right]  ,
\end{align*}
where the constant $C$ depends on $\bar{p}$, $T$, $\bar{\sigma}$,
$\underline{\sigma}$, $L_{1}$, $L_{2}$, $L_{3}$ and $L_{4}$. Since%
\[
\mathbb{\hat{E}}\left[  \sup_{t\leq T}|X_{t}^{\varepsilon}-X_{t}^{\ast}%
|^{p}\right]  \leq \left(  \mathbb{\hat{E}}\left[  \sup_{t\leq T}%
|X_{t}^{\varepsilon}-X_{t}^{\ast}|^{\bar{p}}\right]  \right)  ^{p/\bar{p}},
\]
we obtain (\ref{e2-2}). Noticing $|\phi(X_{T}^{\varepsilon})-\phi(X_{T}^{\ast
})|\leq L_{3}|X_{T}^{\varepsilon}-X_{T}^{\ast}|$ and%
\[
|f(s,X_{s}^{\varepsilon},Y_{s}^{\ast},Z_{s}^{\ast},u_{s}^{\varepsilon
})-f(s,X_{s}^{\ast},Y_{s}^{\ast},Z_{s}^{\ast},u_{s}^{\ast})|\leq L_{3}%
|X_{s}^{\varepsilon}-X_{s}^{\ast}|+L_{4}\varepsilon|u_{s}-u_{s}^{\ast}|,
\]
we obtain%
\begin{align*}
|Y_{t}^{\varepsilon}-Y_{t}^{\ast}|^{p_{1}}  &  \leq C\mathbb{\hat{E}}%
_{t}\left[  \left(  |X_{T}^{\varepsilon}-X_{T}^{\ast}|+\int_{0}^{T}%
(|X_{s}^{\varepsilon}-X_{s}^{\ast}|+\varepsilon|u_{s}-u_{s}^{\ast}|)ds\right)
^{p_{1}}\right] \\
&  \leq C\mathbb{\hat{E}}_{t}\left[  \sup_{s\leq T}|X_{s}^{\varepsilon}%
-X_{s}^{\ast}|^{p_{1}}+\varepsilon^{p_{1}}\left(  \int_{0}^{T}|u_{s}%
-u_{s}^{\ast}|ds\right)  ^{p_{1}}\right]
\end{align*}
by (iii) of Theorem \ref{th1-1}, where $p_{1}=(2+\bar{p})/2$ and the constant
$C$ depends on $\bar{p}$, $T$, $\bar{\sigma}$, $\underline{\sigma}$, $L_{1}$,
$L_{3}$ and $L_{4}$. Then from Theorem \ref{Doob}, we obtain%
\begin{align*}
\mathbb{\hat{E}}\left[  \sup_{t\leq T}|Y_{t}^{\varepsilon}-Y_{t}^{\ast
}|^{p_{1}}\right]   &  \leq C\mathbb{\hat{E}}\left[  \sup_{t\leq
T}\mathbb{\hat{E}}_{t}\left[  \sup_{s\leq T}|X_{s}^{\varepsilon}-X_{s}^{\ast
}|^{p_{1}}+\varepsilon^{p_{1}}\left(  \int_{0}^{T}|u_{s}-u_{s}^{\ast
}|ds\right)  ^{p_{1}}\right]  \right] \\
&  \leq C\left(  \mathbb{\hat{E}}\left[  \left(  \sup_{s\leq T}|X_{s}%
^{\varepsilon}-X_{s}^{\ast}|^{p_{1}}+\varepsilon^{p_{1}}\left(  \int_{0}%
^{T}|u_{s}-u_{s}^{\ast}|ds\right)  ^{p_{1}}\right)  ^{\bar{p}/p_{1}}\right]
\right)  ^{p_{1}/\bar{p}}\\
&  \leq C\left[  \left(  \mathbb{\hat{E}}\left[  \sup_{s\leq T}|X_{s}%
^{\varepsilon}-X_{s}^{\ast}|^{\bar{p}}\right]  \right)  ^{p_{1}/\bar{p}%
}+\varepsilon^{p_{1}}\left(  \mathbb{\hat{E}}\left[  \left(  \int_{0}%
^{T}|u_{t}-u_{t}^{\ast}|^{2}dt\right)  ^{\bar{p}/2}\right]  \right)
^{p_{1}/\bar{p}}\right]  ,
\end{align*}
where the constant $C$ depends on $\bar{p}$, $T$, $\bar{\sigma}$,
$\underline{\sigma}$, $L_{1}$, $L_{3}$ and $L_{4}$. Notice
\[
\mathbb{\hat{E}}\left[  \sup_{t\leq T}|Y_{t}^{\varepsilon}-Y_{t}^{\ast
}|^{p^{\prime}}\right]  \leq \left(  \mathbb{\hat{E}}\left[  \sup_{t\leq
T}|Y_{t}^{\varepsilon}-Y_{t}^{\ast}|^{p_{1}}\right]  \right)  ^{p^{\prime
}/p_{1}}.
\]
Thus we can get (\ref{e2-3}) by (\ref{e2-2}).
\end{proof}

In the following lemma, we will study the estimation of $Z^{\varepsilon}-Z$.
For this purpose, set%
\begin{equation}
\mathcal{P}^{\ast}=\{P\in \mathcal{P}:E_{P}[K_{T}^{\ast}]=0\} \text{ and
}\mathcal{P}_{\varepsilon}^{\ast}=\{P\in \mathcal{P}:E_{P}[K_{T}^{\varepsilon
}]=0\}, \label{e2-4}%
\end{equation}
where $\mathcal{P}$ is defined in Theorem \ref{th-re-1}. Since $K_{T}^{\ast
}\leq0$, we have $K_{T}^{\ast}=0$ $P$-a.s. for each $P\in \mathcal{P}$.
Noticing that $\mathbb{\hat{E}}[K_{T}^{\ast}]=0$ and $\mathcal{P}$ is weakly
compact and convex, we can conclude that $\mathcal{P}^{\ast}$ is a nonempty
weakly compact and convex set.

\begin{lemma}
\label{le2-2}Suppose assumptions (H1)-(H3) hold. Let $u^{\ast}\in
\mathcal{U}[0,T]$ be an optimal control and let $u\in \mathcal{U}[0,T]$ be
given. Then there exists a constant $C$ depending on $\bar{p}$, $T$,
$\bar{\sigma}$, $\underline{\sigma}$, $L_{1}$, $L_{2}$, $L_{3}$, $L_{4}$, $u$
and $u^{\ast}$ such that for any $p\in \lbrack2,(2+\bar{p})/2]$,%
\begin{equation}
\mathbb{\hat{E}}\left[  \left(  \int_{0}^{T}|Z_{t}^{\varepsilon}-Z_{t}^{\ast
}|^{2}dt\right)  ^{p/2}\right]  \leq C\varepsilon^{p/2}, \label{e2-5}%
\end{equation}%
\begin{equation}
E_{P}\left[  \left(  \int_{0}^{T}|Z_{t}^{\varepsilon}-Z_{t}^{\ast}%
|^{2}dt\right)  ^{p/2}\right]  \leq C\varepsilon^{p}\text{, }E_{P_{\varepsilon
}}\left[  \left(  \int_{0}^{T}|Z_{t}^{\varepsilon}-Z_{t}^{\ast}|^{2}dt\right)
^{p/2}\right]  \leq C\varepsilon^{p}, \label{e2-6}%
\end{equation}
for each $P\in \mathcal{P}^{\ast}$ and $P_{\varepsilon}\in \mathcal{P}%
_{\varepsilon}^{\ast}$, where $\bar{p}\in(2,\beta)$ is defined in Theorem
\ref{th1-1}.
\end{lemma}

\begin{proof}
Similar to the proof of Lemma \ref{le2-1}, we only need to prove the case
$p_{1}=(2+\bar{p})/2$. By Proposition 3.8 in \cite{HJPS1}, we have%
\[
\mathbb{\hat{E}}\left[  \left(  \int_{0}^{T}|Z_{t}^{\varepsilon}-Z_{t}^{\ast
}|^{2}dt\right)  ^{p_{1}/2}\right]  \leq C\left \{  \mathbb{\hat{E}}\left[
\sup_{t\leq T}|Y_{t}^{\varepsilon}-Y_{t}^{\ast}|^{p_{1}}\right]
+\Lambda^{1/2}\left(  \mathbb{\hat{E}}\left[  \sup_{t\leq T}|Y_{t}%
^{\varepsilon}-Y_{t}^{\ast}|^{p_{1}}\right]  \right)  ^{1/2}\right \}  ,
\]
where the constant $C$ depends on $\bar{p}$, $T$, $\underline{\sigma}$ and
$L_{1}$,%
\[
\Lambda=\mathbb{\hat{E}}\left[  \sup_{t\leq T}(|Y_{t}^{\varepsilon}|^{p_{1}%
}+|Y_{t}^{\ast}|^{p_{1}})+\left(  \int_{0}^{T}(|f(t,X_{t}^{\varepsilon
},0,0,u_{t}^{\varepsilon})|+|f(t,X_{t}^{\ast},0,0,u_{t}^{\ast})|)dt\right)
^{p_{1}}\right]  .
\]
One can check that $\Lambda$ is bounded by (\ref{e2-2}) and (\ref{e2-3}). Then
we obtain (\ref{e2-5}) by (\ref{e2-3}). For each $P\in \mathcal{P}^{\ast}$,
applying It\^{o}'s formula to $|Y_{t}^{\varepsilon}-Y_{t}^{\ast}|^{2}$ under
$P$, we obtain%
\[
\int_{0}^{T}|\hat{Z}_{t}^{\varepsilon}|^{2}d\langle B\rangle_{t}\leq|\hat
{Y}_{T}^{\varepsilon}|^{2}+2\int_{0}^{T}|\hat{Y}_{t}^{\varepsilon
}(f_{\varepsilon}(t)-f(t))|dt+2\left \vert \int_{0}^{T}\hat{Y}_{t}%
^{\varepsilon}dK_{t}^{\varepsilon}\right \vert +2\left \vert \int_{0}^{T}\hat
{Y}_{t}^{\varepsilon}\hat{Z}_{t}^{\varepsilon}dB_{t}\right \vert ,
\]
where $\hat{Z}_{t}^{\varepsilon}=Z_{t}^{\varepsilon}-Z_{t}^{\ast}$, $\hat
{Y}_{t}^{\varepsilon}=Y_{t}^{\varepsilon}-Y_{t}^{\ast}$, $f_{\varepsilon
}(t)=f(t,X_{t}^{\varepsilon},Y_{t}^{\varepsilon},Z_{t}^{\varepsilon}%
,u_{t}^{\varepsilon})$, $f(t)=f(t,X_{t}^{\ast},Y_{t}^{\ast},Z_{t}^{\ast}%
,u_{t}^{\ast})$, $\langle B\rangle$ is the quadratic variation process of $B$
satisfying $\underline{\sigma}^{2}dt\leq d\langle B\rangle_{t}\leq \bar{\sigma
}^{2}dt$ q.s. (see Corollary 3.5.5 in \cite{P2019}). Under the assumption
(H2), one can deduce%
\begin{align*}
&  E_{P}\left[  \left(  \int_{0}^{T}|\hat{Z}_{t}^{\varepsilon}|^{2}dt\right)
^{p_{1}/2}\right] \\
&  \leq CE_{P}\left[  \sup_{t\leq T}(|\hat{X}_{t}^{\varepsilon}|^{p_{1}}%
+|\hat{Y}_{t}^{\varepsilon}|^{p_{1}})+\varepsilon^{p_{1}}\left(  \int_{0}%
^{T}|u_{t}-u_{t}^{\ast}|^{2}dt\right)  ^{p_{1}/2}+|K_{T}^{\varepsilon}%
|^{p_{1}/2}\sup_{t\leq T}|\hat{Y}_{t}^{\varepsilon}|^{p_{1}/2}\right] \\
&  \leq C\mathbb{\hat{E}}\left[  \sup_{t\leq T}(|\hat{X}_{t}^{\varepsilon
}|^{p_{1}}+|\hat{Y}_{t}^{\varepsilon}|^{p_{1}})+\varepsilon^{p_{1}}\left(
\int_{0}^{T}|u_{t}-u_{t}^{\ast}|^{2}dt\right)  ^{p_{1}/2}\right]
+CE_{P}\left[  |K_{T}^{\varepsilon}|^{p_{1}/2}\sup_{t\leq T}|\hat{Y}%
_{t}^{\varepsilon}|^{p_{1}/2}\right]  ,
\end{align*}
where $\hat{X}_{t}^{\varepsilon}=X_{t}^{\varepsilon}-X_{t}^{\ast}$, the
constant $C$ depends on $\bar{p}$, $T$, $\underline{\sigma}$, $L_{1}$, $L_{3}$
and $L_{4}$. It follows from Lemma \ref{le2-1} that
\begin{equation}
E_{P}\left[  \left(  \int_{0}^{T}|\hat{Z}_{t}^{\varepsilon}|^{2}dt\right)
^{p_{1}/2}\right]  \leq C\varepsilon^{p_{1}}+CE_{P}\left[  |K_{T}%
^{\varepsilon}|^{p_{1}/2}\sup_{t\leq T}|\hat{Y}_{t}^{\varepsilon}|^{p_{1}%
/2}\right]  , \label{e2-7}%
\end{equation}
where the constant $C$ depends on $\bar{p}$, $T$, $\bar{\sigma}$,
$\underline{\sigma}$, $L_{1}$, $L_{2}$, $L_{3}$, $L_{4}$, $u$ and $u^{\ast}$.
On the other hand,%
\[
K_{T}^{\varepsilon}=\hat{Y}_{T}^{\varepsilon}-\hat{Y}_{0}^{\varepsilon}%
-\int_{0}^{T}(f_{\varepsilon}(t)-f(t))dt-\int_{0}^{T}\hat{Z}_{t}^{\varepsilon
}dB_{t}.
\]
From this and Lemma \ref{le2-1}, we can obtain%
\begin{equation}
E_{P}\left[  |K_{T}^{\varepsilon}|^{p_{1}}\right]  \leq C\varepsilon^{p_{1}%
}+CE_{P}\left[  \left(  \int_{0}^{T}|\hat{Z}_{t}^{\varepsilon}|^{2}dt\right)
^{p_{1}/2}\right]  , \label{e2-8}%
\end{equation}
where the constant $C$ depends on $\bar{p}$, $T$, $\bar{\sigma}$,
$\underline{\sigma}$, $L_{1}$, $L_{2}$, $L_{3}$, $L_{4}$, $u$ and $u^{\ast}$.
By (\ref{e2-7}), (\ref{e2-8}), $E_{P}\left[  \sup_{t\leq T}|\hat{Y}%
_{t}^{\varepsilon}|^{p_{1}}\right]  \leq \mathbb{\hat{E}}\left[  \sup_{t\leq
T}|\hat{Y}_{t}^{\varepsilon}|^{p_{1}}\right]  $ and $2ab\leq \lambda
a^{2}+\lambda^{-1}b^{2}$ for any $\lambda>0$, we can get $E_{P}\left[  \left(
\int_{0}^{T}|\hat{Z}_{t}^{\varepsilon}|^{2}dt\right)  ^{p_{1}/2}\right]  \leq
C\varepsilon^{p_{1}}$, where the constant $C$ depends on $\bar{p}$, $T$,
$\bar{\sigma}$, $\underline{\sigma}$, $L_{1}$, $L_{2}$, $L_{3}$, $L_{4}$, $u$
and $u^{\ast}$. Applying the same analysis to any $P_{\varepsilon}%
\in \mathcal{P}_{\varepsilon}^{\ast}$, we can obtain (\ref{e2-6}).
\end{proof}

Next, we present the variational inequality for cost functional, which is a
necessary condition for optimality. For this purpose, let $d(P,Q)$ denote the
Prohorov metric on $\mathcal{P}$. It is well known that $d(P_{n}%
,P)\rightarrow0$ as $n\rightarrow \infty$ if and only if $\{P_{n}:n\geq1\}$
converges weakly to $P$. For each $\alpha>0$, set%
\begin{equation}
\mathcal{P}^{\ast \alpha}=\{P\in \mathcal{P}:d(P,\mathcal{P}^{\ast})\leq
\alpha \},\label{e2-9}%
\end{equation}
where $d(P,\mathcal{P}^{\ast}):=\inf \{d(P,Q):Q\in \mathcal{P}^{\ast}\}$.
Consider the following adjoint equation under $P\in \mathcal{P}$:%
\begin{equation}
\left \{
\begin{array}
[c]{l}%
dl_{t}^{P}=[-f_{y}(t)l_{t}^{P}+b_{y}(t)p_{t}^{P}+\gamma_{t}\sigma_{y}%
(t)q_{t}^{P}]dt-\gamma_{t}^{-1}f_{z}(t)l_{t}^{P}dB_{t},\\
dp_{t}^{P}=[f_{x}(t)l_{t}^{P}-b_{x}(t)p_{t}^{P}-\gamma_{t}\sigma_{x}%
(t)q_{t}^{P}]dt+q_{t}^{P}dB_{t}+dN_{t}^{P},\\
l_{0}^{P}=1,\text{ }p_{T}^{P}=\phi_{x}(X_{T}^{\ast})l_{T}^{p},
\end{array}
\right.  \label{e2-10}%
\end{equation}
where $f_{y}(t)=f_{y}(t,X_{t}^{\ast},Y_{t}^{\ast},Z_{t}^{\ast},u_{t}^{\ast})$,
similar definitions for $b_{y}(t)$, $\sigma_{y}(t)$, $f_{z}(t)$, $f_{x}(t)$,
$b_{x}(t)$ and $\sigma_{x}(t)$, $d\langle B\rangle_{t}=\gamma_{t}dt$ q.s.,
$N^{P}$ is a square integrable martingale under $P$ such that $N^{P}$ is
orthogonal to $B$ and $N_{0}^{P}=0$. It follows from Theorem 6.1 in \cite{EH}
(see also the proof of Theorem 4.1 in \cite{Li}) that the adjoint equation
(\ref{e2-10}) under $P\in \mathcal{P}$ has a unique $(l^{P}$,$p^{P}$,$q^{P}%
$,$N^{P})$ satisfying $l^{P}$, $p^{P}\in S_{P}^{2}(0,T)$ and $q^{P}\in
M_{P}^{2}(0,T)$, where $S_{P}^{2}(0,T)$ is the completion of $S_{G}^{0}(0,T)$
under the norm $\left(  E_{P}[\sup_{t\leq T}|\eta_{t}|^{2}]\right)  ^{1/2}$
and $M_{P}^{2}(0,T)$ is the completion of $M_{G}^{0}(0,T)$ under the norm
$\left(  E_{P}\left[  \int_{0}^{T}|\eta_{t}|^{2}dt\right]  \right)  ^{1/2}$.
Furthermore, we have $l_{t}^{P}\geq0$ $P$-a.s. and%
\begin{equation}
\sup_{P\in \mathcal{P}}E_{P}\left[  \sup_{t\leq T}(|l_{t}^{P}|^{2}+|p_{t}%
^{P}|^{2})+\int_{0}^{T}|q_{t}^{P}|^{2}dt\right]  \leq C,\label{e2-11}%
\end{equation}
where the constant $C$ depends on $T$, $\bar{\sigma}$, $\underline{\sigma}$,
$L_{1}$, $L_{2}$ and $L_{3}$.

\begin{theorem}
\label{th2-3}Suppose assumptions (H1)-(H4) hold. Let $u^{\ast}\in
\mathcal{U}[0,T]$ be an optimal control. Then for each given $u\in
\mathcal{U}[0,T]$, we have%
\begin{equation}
\inf_{\alpha>0}\sup_{P\in \mathcal{P}^{\ast \alpha}}E_{P}\left[  \int_{0}%
^{T}(b_{v}(t)p_{t}^{P}+\gamma_{t}\sigma_{v}(t)q_{t}^{P}-f_{v}(t)l_{t}%
^{P})(u_{t}-u_{t}^{\ast})dt\right]  \geq0, \label{e2-new-12}%
\end{equation}
where $b_{v}(t)=(b_{v_{1}}(t,X_{t}^{\ast},Y_{t}^{\ast},u_{t}^{\ast}%
),\ldots,b_{v_{m}}(t,X_{t}^{\ast},Y_{t}^{\ast},u_{t}^{\ast}))$, similar
definitions for $\sigma_{v}(t)$ and $f_{v}(t)$, $\mathcal{P}^{\ast \alpha}$ is
defined in (\ref{e2-9}), $(l^{P}$,$p^{P}$,$q^{P}$,$N^{P})$ is the solution of
the adjoint equation (\ref{e2-10}) under $P$.
\end{theorem}

\begin{proof}
Set $\hat{X}_{t}^{\varepsilon}=X_{t}^{\varepsilon}-X_{t}^{\ast}$, $\hat{Y}%
_{t}^{\varepsilon}=Y_{t}^{\varepsilon}-Y_{t}^{\ast}$, $\hat{Z}_{t}%
^{\varepsilon}=Z_{t}^{\varepsilon}-Z_{t}^{\ast}$ for $\varepsilon \in(0,1]$, we
have%
\begin{equation}
\left \{
\begin{array}
[c]{rl}%
d\hat{X}_{t}^{\varepsilon}= & [b_{x}(t)\hat{X}_{t}^{\varepsilon}+b_{y}%
(t)\hat{Y}_{t}^{\varepsilon}+b_{v}(t)\varepsilon(u_{t}-u_{t}^{\ast}%
)+I_{b}^{\varepsilon}(t)]dt\\
& +[\sigma_{x}(t)\hat{X}_{t}^{\varepsilon}+\sigma_{y}(t)\hat{Y}_{t}%
^{\varepsilon}+\sigma_{v}(t)\varepsilon(u_{t}-u_{t}^{\ast})+I_{\sigma
}^{\varepsilon}(t)]dB_{t},\\
d\hat{Y}_{t}^{\varepsilon}= & [f_{x}(t)\hat{X}_{t}^{\varepsilon}+f_{y}%
(t)\hat{Y}_{t}^{\varepsilon}+f_{z}(t)\hat{Z}_{t}^{\varepsilon}+f_{v}%
(t)\varepsilon(u_{t}-u_{t}^{\ast})+I_{f}^{\varepsilon}(t)]dt+\hat{Z}%
_{t}^{\varepsilon}dB_{t}+dK_{t}^{\varepsilon}-dK_{t}^{\ast},\\
\hat{X}_{0}^{\varepsilon}= & 0\text{, }\hat{Y}_{T}^{\varepsilon}=\phi
_{x}(X_{T}^{\ast})\hat{X}_{T}^{\varepsilon}+I_{\phi}^{\varepsilon}(T),
\end{array}
\right.  \label{e2-12}%
\end{equation}
where $b_{x}(\lambda,t)=b_{x}(t,X_{t}^{\ast}+\lambda \hat{X}_{t}^{\varepsilon
},Y_{t}^{\ast}+\lambda \hat{Y}_{t}^{\varepsilon},u_{t}^{\ast}+\lambda
\varepsilon(u_{t}-u_{t}^{\ast}))$ for $\lambda \in \lbrack0,1]$ and
$b_{x}(t)=b_{x}(0,t)$, similar definitions for $b_{y}(\lambda,t)$,
$b_{v}(\lambda,t)$, $\sigma_{x}(\lambda,t)$, $\sigma_{y}(\lambda,t)$,
$\sigma_{v}(\lambda,t)$, $f_{x}(\lambda,t)$, $f_{y}(\lambda,t)$,
$f_{z}(\lambda,t)$ and $f_{v}(\lambda,t)$,%
\[%
\begin{array}
[c]{rl}%
I_{b}^{\varepsilon}(t)= & \int_{0}^{1}[b_{x}(\lambda,t)-b_{x}(t)]d\lambda
\hat{X}_{t}^{\varepsilon}+\int_{0}^{1}[b_{y}(\lambda,t)-b_{y}(t)]d\lambda
\hat{Y}_{t}^{\varepsilon}+\int_{0}^{1}[b_{v}(\lambda,t)-b_{v}(t)]d\lambda
\varepsilon(u_{t}-u_{t}^{\ast}),\\
I_{\sigma}^{\varepsilon}(t)= & \int_{0}^{1}[\sigma_{x}(\lambda,t)-\sigma
_{x}(t)]d\lambda \hat{X}_{t}^{\varepsilon}+\int_{0}^{1}[\sigma_{y}%
(\lambda,t)-\sigma_{y}(t)]d\lambda \hat{Y}_{t}^{\varepsilon}+\int_{0}%
^{1}[\sigma_{v}(\lambda,t)-\sigma_{v}(t)]d\lambda \varepsilon(u_{t}-u_{t}%
^{\ast}),\\
I_{f}^{\varepsilon}(t)= & \int_{0}^{1}[f_{x}(\lambda,t)-f_{x}(t)]d\lambda
\hat{X}_{t}^{\varepsilon}+\int_{0}^{1}[f_{y}(\lambda,t)-f_{y}(t)]d\lambda
\hat{Y}_{t}^{\varepsilon}+\int_{0}^{1}[f_{z}(\lambda,t)-f_{z}(t)]d\lambda
\hat{Z}_{t}^{\varepsilon}\\
& +\int_{0}^{1}[f_{v}(\lambda,t)-f_{v}(t)]d\lambda \varepsilon(u_{t}%
-u_{t}^{\ast}),\\
I_{\phi}^{\varepsilon}(T)= & \int_{0}^{1}[\phi_{x}(X_{T}^{\ast}+\lambda \hat
{X}_{T}^{\varepsilon})-\phi_{x}(X_{T}^{\ast})]d\lambda \hat{X}_{T}%
^{\varepsilon}.
\end{array}
\]
For each given $P_{\varepsilon}\in \mathcal{P}_{\varepsilon}^{\ast}$, applying
It\^{o}'s formula to $l_{t}^{P_{\varepsilon}}\hat{Y}_{t}^{\varepsilon}%
-p_{t}^{P_{\varepsilon}}\hat{X}_{t}^{\varepsilon}$ under $P_{\varepsilon}$, we
can get%
\begin{equation}%
\begin{array}
[c]{rl}%
\hat{Y}_{0}^{\varepsilon}= & \varepsilon E_{P_{\varepsilon}}\left[  \int
_{0}^{T}(b_{v}(t)p_{t}^{P_{\varepsilon}}+\gamma_{t}\sigma_{v}(t)q_{t}%
^{P_{\varepsilon}}-f_{v}(t)l_{t}^{P_{\varepsilon}})(u_{t}-u_{t}^{\ast
})dt\right] \\
& +E_{P_{\varepsilon}}\left[  l_{T}^{P_{\varepsilon}}I_{\phi}^{\varepsilon
}(T)+\int_{0}^{T}(p_{t}^{P_{\varepsilon}}I_{b}^{\varepsilon}(t)+\gamma
_{t}q_{t}^{P_{\varepsilon}}I_{\sigma}^{\varepsilon}(t)-l_{t}^{P_{\varepsilon}%
}I_{f}^{\varepsilon}(t))dt+\int_{0}^{T}l_{t}^{P_{\varepsilon}}dK_{t}^{\ast
}\right]  .
\end{array}
\label{e2-13}%
\end{equation}
In the following proof, we first prove%
\begin{equation}
\lim_{\varepsilon \downarrow0}\varepsilon^{-1}E_{P_{\varepsilon}}\left[
\left \vert l_{T}^{P_{\varepsilon}}I_{\phi}^{\varepsilon}(T)+\int_{0}^{T}%
(p_{t}^{P_{\varepsilon}}I_{b}^{\varepsilon}(t)+\gamma_{t}q_{t}^{P_{\varepsilon
}}I_{\sigma}^{\varepsilon}(t)-l_{t}^{P_{\varepsilon}}I_{f}^{\varepsilon
}(t))dt\right \vert \right]  =0. \label{e2-14}%
\end{equation}
Under the assumptions (H2) and (H4), we have $|I_{\phi}^{\varepsilon}%
(T)|\leq \lbrack \rho(|\hat{X}_{T}^{\varepsilon}|)\wedge C]|\hat{X}%
_{T}^{\varepsilon}|$,
\[
|I_{b}^{\varepsilon}(t)|+|I_{\sigma}^{\varepsilon}(t)|\leq \lbrack \rho(|\hat
{X}_{t}^{\varepsilon}|+|\hat{Y}_{t}^{\varepsilon}|+\varepsilon|u_{t}%
-u_{t}^{\ast}|)\wedge C](|\hat{X}_{t}^{\varepsilon}|+|\hat{Y}_{t}%
^{\varepsilon}|+\varepsilon|u_{t}-u_{t}^{\ast}|),
\]%
\[
|I_{f}^{\varepsilon}(t)|\leq \lbrack \rho(|\hat{X}_{t}^{\varepsilon}|+|\hat
{Y}_{t}^{\varepsilon}|+|\hat{Z}_{t}^{\varepsilon}|+\varepsilon|u_{t}%
-u_{t}^{\ast}|)\wedge C](|\hat{X}_{t}^{\varepsilon}|+|\hat{Y}_{t}%
^{\varepsilon}|+|\hat{Z}_{t}^{\varepsilon}|+\varepsilon|u_{t}-u_{t}^{\ast}|),
\]
where the constant $C$ depends on $L_{1}$, $L_{2}$, $L_{3}$ and $L_{4}$. For
each fixed $\varepsilon_{0}>0$, we have $\rho(|\hat{X}_{T}^{\varepsilon
}|)\wedge C\leq \rho(\varepsilon_{0})+CI_{\{|\hat{X}_{T}^{\varepsilon
}|>\varepsilon_{0}\}}$,%
\[
\rho(|\hat{X}_{t}^{\varepsilon}|+|\hat{Y}_{t}^{\varepsilon}|+\varepsilon
|u_{t}-u_{t}^{\ast}|)\wedge C\leq \rho(3\varepsilon_{0})+C\left \{
I_{\{|\hat{X}_{t}^{\varepsilon}|>\varepsilon_{0}\}}+I_{\{|\hat{Y}%
_{t}^{\varepsilon}|>\varepsilon_{0}\}}+I_{\{ \varepsilon|u_{t}-u_{t}^{\ast
}|>\varepsilon_{0}\}}\right \}  ,
\]%
\[
\rho(|\hat{X}_{t}^{\varepsilon}|+|\hat{Y}_{t}^{\varepsilon}|+|\hat{Z}%
_{t}^{\varepsilon}|+\varepsilon|u_{t}-u_{t}^{\ast}|)\wedge C\leq
\rho(4\varepsilon_{0})+C\left \{  I_{\{|\hat{X}_{t}^{\varepsilon}%
|>\varepsilon_{0}\}}+I_{\{|\hat{Y}_{t}^{\varepsilon}|>\varepsilon_{0}%
\}}+I_{\{|\hat{Z}_{t}^{\varepsilon}|>\varepsilon_{0}\}}+I_{\{ \varepsilon
|u_{t}-u_{t}^{\ast}|>\varepsilon_{0}\}}\right \}  .
\]
Set $p_{1}=(2+\bar{p})/2$, where $\bar{p}\in(2,\beta)$ is defined in Theorem
\ref{th1-1}. By (\ref{e2-11}) and Lemmas \ref{le2-1} and \ref{le2-2}, we have
\begin{align*}
E_{P_{\varepsilon}}\left[  \left \vert l_{T}^{P_{\varepsilon}}I_{\phi
}^{\varepsilon}(T)\right \vert \right]   &  \leq \rho(\varepsilon_{0}%
)E_{P_{\varepsilon}}\left[  l_{T}^{P_{\varepsilon}}|\hat{X}_{T}^{\varepsilon
}|\right]  +CE_{P_{\varepsilon}}\left[  l_{T}^{P_{\varepsilon}}|\hat{X}%
_{T}^{\varepsilon}|I_{\{|\hat{X}_{T}^{\varepsilon}|>\varepsilon_{0}\}}\right]
\\
&  \leq \left(  E_{P_{\varepsilon}}\left[  |l_{T}^{P_{\varepsilon}}%
|^{2}\right]  \right)  ^{1/2}\left[  \rho(\varepsilon_{0})\left(
E_{P_{\varepsilon}}\left[  |\hat{X}_{T}^{\varepsilon}|^{2}\right]  \right)
^{1/2}+C\left(  E_{P_{\varepsilon}}\left[  |\hat{X}_{T}^{\varepsilon}%
|^{2}I_{\{|\hat{X}_{T}^{\varepsilon}|>\varepsilon_{0}\}}\right]  \right)
^{1/2}\right] \\
&  \leq C\rho(\varepsilon_{0})\varepsilon+C\left(  E_{P_{\varepsilon}}\left[
|\hat{X}_{T}^{\varepsilon}|^{2}I_{\{|\hat{X}_{T}^{\varepsilon}|>\varepsilon
_{0}\}}\right]  \right)  ^{1/2}\\
&  \leq C\rho(\varepsilon_{0})\varepsilon+C\left(  E_{P_{\varepsilon}}\left[
|\hat{X}_{T}^{\varepsilon}|^{2}|\hat{X}_{T}^{\varepsilon}|^{p_{1}%
-2}\varepsilon_{0}^{2-p_{1}}\right]  \right)  ^{1/2}\\
&  \leq C\rho(\varepsilon_{0})\varepsilon+C\varepsilon_{0}^{(2-p_{1}%
)/2}\varepsilon^{p_{1}/2},
\end{align*}%
\begin{align*}
E_{P_{\varepsilon}}\left[  \int_{0}^{T}|q_{t}^{P_{\varepsilon}}||\hat{X}%
_{t}^{\varepsilon}|I_{\{|\hat{Y}_{t}^{\varepsilon}|>\varepsilon_{0}%
\}}dt\right]   &  \leq \left(  E_{P_{\varepsilon}}\left[  \int_{0}^{T}%
|q_{t}^{P_{\varepsilon}}|^{2}dt\right]  \right)  ^{1/2}\left(
E_{P_{\varepsilon}}\left[  \int_{0}^{T}|\hat{X}_{t}^{\varepsilon}%
|^{2}I_{\{|\hat{Y}_{t}^{\varepsilon}|>\varepsilon_{0}\}}dt\right]  \right)
^{1/2}\\
&  \leq C\left(  E_{P_{\varepsilon}}\left[  \int_{0}^{T}|\hat{X}%
_{t}^{\varepsilon}|^{p_{1}}dt\right]  \right)  ^{1/p_{1}}\left(
E_{P_{\varepsilon}}\left[  \int_{0}^{T}I_{\{|\hat{Y}_{t}^{\varepsilon
}|>\varepsilon_{0}\}}dt\right]  \right)  ^{(p_{1}-2)/2p_{1}}\\
&  \leq C\varepsilon \left(  E_{P_{\varepsilon}}\left[  \int_{0}^{T}|\hat
{Y}_{t}^{\varepsilon}|^{2}\varepsilon_{0}^{-2}dt\right]  \right)
^{(p_{1}-2)/2p_{1}}\\
&  \leq C\varepsilon_{0}^{(2-p_{1})/p_{1}}\varepsilon^{(2p_{1}-2)/p_{1}},
\end{align*}%
\begin{align*}
E_{P_{\varepsilon}}\left[  \varepsilon \int_{0}^{T}|q_{t}^{P_{\varepsilon}%
}||u_{t}-u_{t}^{\ast}|I_{\{|\hat{Y}_{t}^{\varepsilon}|>\varepsilon_{0}%
\}}dt\right]   &  \leq C\varepsilon \left(  E_{P_{\varepsilon}}\left[  \int
_{0}^{T}|u_{t}-u_{t}^{\ast}|^{2}I_{\{|\hat{Y}_{t}^{\varepsilon}|>\varepsilon
_{0}\}}dt\right]  \right)  ^{1/2}\\
&  \leq C\varepsilon \left(  \mathbb{\hat{E}}\left[  \int_{0}^{T}|u_{t}%
-u_{t}^{\ast}|^{2}I_{\{|\hat{Y}_{t}^{\varepsilon}|>\varepsilon_{0}%
\}}dt\right]  \right)  ^{1/2},
\end{align*}%
\begin{align*}
E_{P_{\varepsilon}}\left[  \int_{0}^{T}l_{t}^{P_{\varepsilon}}|\hat{Z}%
_{t}^{\varepsilon}|I_{\{|\hat{Z}_{t}^{\varepsilon}|>\varepsilon_{0}%
\}}dt\right]   &  \leq \left(  E_{P_{\varepsilon}}\left[  \sup_{t\leq T}%
|l_{t}^{P_{\varepsilon}}|^{2}\right]  \right)  ^{1/2}\left(  E_{P_{\varepsilon
}}\left[  \left(  \int_{0}^{T}|\hat{Z}_{t}^{\varepsilon}|I_{\{|\hat{Z}%
_{t}^{\varepsilon}|>\varepsilon_{0}\}}dt\right)  ^{2}\right]  \right)
^{1/2}\\
&  \leq C\left(  E_{P_{\varepsilon}}\left[  \left(  \int_{0}^{T}|\hat{Z}%
_{t}^{\varepsilon}|^{2}dt\right)  \left(  \int_{0}^{T}I_{\{|\hat{Z}%
_{t}^{\varepsilon}|>\varepsilon_{0}\}}dt\right)  \right]  \right)  ^{1/2}\\
&  \leq C\left(  E_{P_{\varepsilon}}\left[  \left(  \int_{0}^{T}|\hat{Z}%
_{t}^{\varepsilon}|^{2}dt\right)  ^{p_{1}/2}\right]  \right)  ^{1/p_{1}%
}\left(  E_{P_{\varepsilon}}\left[  \int_{0}^{T}I_{\{|\hat{Z}_{t}%
^{\varepsilon}|>\varepsilon_{0}\}}dt\right]  \right)  ^{(p_{1}-2)/2p_{1}}\\
&  \leq C\varepsilon \left(  E_{P_{\varepsilon}}\left[  \int_{0}^{T}|\hat
{Z}_{t}^{\varepsilon}|^{2}\varepsilon_{0}^{-2}dt\right]  \right)
^{(p_{1}-2)/2p_{1}}\\
&  \leq C\varepsilon_{0}^{(2-p_{1})/p_{1}}\varepsilon^{(2p_{1}-2)/p_{1}},
\end{align*}
where the constant $C$ depends on $\bar{p}$, $T$, $\bar{\sigma}$,
$\underline{\sigma}$, $L_{1}$, $L_{2}$, $L_{3}$, $L_{4}$, $u$ and $u^{\ast}$.
The other terms in (\ref{e2-14}) can be estimated using a similar method, and
thus we can obtain%
\begin{equation}%
\begin{array}
[c]{l}%
\varepsilon^{-1}E_{P_{\varepsilon}}\left[  \left \vert l_{T}^{P_{\varepsilon}%
}I_{\phi}^{\varepsilon}(T)+\int_{0}^{T}(p_{t}^{P_{\varepsilon}}I_{b}%
^{\varepsilon}(t)+\gamma_{t}q_{t}^{P_{\varepsilon}}I_{\sigma}^{\varepsilon
}(t)-l_{t}^{P_{\varepsilon}}I_{f}^{\varepsilon}(t))dt\right \vert \right] \\
\leq C[\rho(\varepsilon_{0})+\rho(3\varepsilon_{0})+\rho(4\varepsilon
_{0})]+C\varepsilon_{0}^{(2-p_{1})/2}\varepsilon^{(p_{1}-2)/2}+C\varepsilon
_{0}^{(2-p_{1})/p_{1}}\varepsilon^{(p_{1}-2)/p_{1}}\\
\  \  \ +C\left(  \mathbb{\hat{E}}\left[  \int_{0}^{T}|u_{t}-u_{t}^{\ast}%
|^{2}I_{\{|\hat{X}_{t}^{\varepsilon}|>\varepsilon_{0}\} \cup \{|\hat{Y}%
_{t}^{\varepsilon}|>\varepsilon_{0}\} \cup \{|\hat{Z}_{t}^{\varepsilon
}|>\varepsilon_{0}\} \cup \{ \varepsilon|u_{t}-u_{t}^{\ast}|>\varepsilon_{0}%
\}}dt\right]  \right)  ^{1/2},
\end{array}
\label{e2-15}%
\end{equation}
where the constant $C$ depends on $\bar{p}$, $T$, $\bar{\sigma}$,
$\underline{\sigma}$, $L_{1}$, $L_{2}$, $L_{3}$, $L_{4}$, $u$ and $u^{\ast}$.
Since $u-u^{\ast}\in M_{G}^{2}(0,T;\mathbb{R}^{m})$ and%
\begin{align*}
&  \mathbb{\hat{E}}\left[  \int_{0}^{T}I_{\{|\hat{X}_{t}^{\varepsilon
}|>\varepsilon_{0}\} \cup \{|\hat{Y}_{t}^{\varepsilon}|>\varepsilon_{0}\}
\cup \{|\hat{Z}_{t}^{\varepsilon}|>\varepsilon_{0}\} \cup \{ \varepsilon
|u_{t}-u_{t}^{\ast}|>\varepsilon_{0}\}}dt\right] \\
&  \leq \mathbb{\hat{E}}\left[  \int_{0}^{T}\varepsilon_{0}^{-2}(|\hat{X}%
_{t}^{\varepsilon}|^{2}+|\hat{Y}_{t}^{\varepsilon}|^{2}+|\hat{Z}%
_{t}^{\varepsilon}|^{2}+\varepsilon^{2}|u_{t}-u_{t}^{\ast}|^{2})dt\right]
\rightarrow0
\end{align*}
as $\varepsilon \downarrow0$, we get $\mathbb{\hat{E}}\left[  \int_{0}%
^{T}|u_{t}-u_{t}^{\ast}|^{2}I_{\{|\hat{X}_{t}^{\varepsilon}|>\varepsilon_{0}\}
\cup \{|\hat{Y}_{t}^{\varepsilon}|>\varepsilon_{0}\} \cup \{|\hat{Z}%
_{t}^{\varepsilon}|>\varepsilon_{0}\} \cup \{ \varepsilon|u_{t}-u_{t}^{\ast
}|>\varepsilon_{0}\}}dt\right]  \rightarrow0$ as $\varepsilon \downarrow0$ by
Lemma 4.2 in \cite{HJ1}. Therefore, we obtain (\ref{e2-14}) by first taking
$\varepsilon \downarrow0$ and then $\varepsilon_{0}\downarrow0$ in
(\ref{e2-15}). Noticing $\int_{0}^{T}l_{t}^{P_{\varepsilon}}dK_{t}^{\ast}%
\leq0$ and $\hat{Y}_{0}^{\varepsilon}\geq0$, we obtain%
\begin{equation}
\liminf_{\varepsilon \downarrow0}E_{P_{\varepsilon}}\left[  \int_{0}^{T}%
(b_{v}(t)p_{t}^{P_{\varepsilon}}+\gamma_{t}\sigma_{v}(t)q_{t}^{P_{\varepsilon
}}-f_{v}(t)l_{t}^{P_{\varepsilon}})(u_{t}-u_{t}^{\ast})dt\right]  \geq0
\label{e2-16}%
\end{equation}
by (\ref{e2-13}) and (\ref{e2-14}). It follows from Lemmas \ref{le2-1} and
\ref{le2-2} that $\mathbb{\hat{E}}[|K_{t}^{\varepsilon}-K_{t}^{\ast
}|]\rightarrow0$ as $\varepsilon \downarrow0$. Then we obtain%
\begin{equation}
|E_{P_{\varepsilon}}[K_{t}^{\ast}]|=E_{P_{\varepsilon}}[K_{t}^{\varepsilon
}-K_{t}^{\ast}]\leq \mathbb{\hat{E}}[|K_{t}^{\varepsilon}-K_{t}^{\ast
}|]\rightarrow0 \label{e2-17}%
\end{equation}
as $\varepsilon \downarrow0$. Since $\mathcal{P}$ is weakly compact, there
exists a sequence $(\varepsilon_{n})_{n\geq1}$ such that $\varepsilon
_{n}\downarrow0$ as $n\rightarrow \infty$ and $\{P_{\varepsilon_{n}}:n\geq1\}$
converges weakly to $P\in \mathcal{P}$. By Lemma 29 in \cite{DHP11} and
(\ref{e2-17}), we have%
\[
E_{P}[K_{t}^{\ast}]=\lim_{n\rightarrow \infty}E_{P_{\varepsilon_{n}}}%
[K_{t}^{\ast}]=0,
\]
which implies $P\in \mathcal{P}^{\ast}$. Thus we obtain (\ref{e2-new-12}) by
(\ref{e2-16}) and $d(P_{\varepsilon_{n}},\mathcal{P}^{\ast})\rightarrow0$ as
$n\rightarrow \infty$.
\end{proof}

At the end of this section, we consider the following linear control system:%
\begin{equation}
\left \{
\begin{array}
[c]{rl}%
dX_{t}^{u}= & (a_{1}(t)X_{t}^{u}+b_{1}(t)Y_{t}^{u}+c_{1}(t)u_{t}+e_{1}(t))dt\\
& +(a_{2}(t)X_{t}^{u}+b_{2}(t)Y_{t}^{u}+c_{2}(t)u_{t}+e_{2}(t))dB_{t},\\
dY_{t}^{u}= & (a_{3}(t)X_{t}^{u}+b_{3}(t)Y_{t}^{u}+c_{3}(t)u_{t}%
+e_{3}(t))dt+Z_{t}^{u}dB_{t}+dK_{t}^{u},\\
X_{0}^{u}= & x_{0}\in \mathbb{R},\text{ }Y_{T}^{u}=A_{1}X_{T}^{u}+A_{2},\text{
}t\in \lbrack0,T],
\end{array}
\right.  \label{e2-18}%
\end{equation}
where $A_{1}$, $A_{2}\in \mathbb{R}$, $a_{i}$, $b_{i}$, $e_{i}:[0,T]\rightarrow
\mathbb{R}$, $c_{i}:[0,T]\rightarrow \mathbb{R}^{1\times m}$, for $i\leq3$.

\begin{proposition}
\label{pr2-4}Suppose assumptions (H1)-(H4) hold. Let $u^{\ast}\in
\mathcal{U}[0,T]$ be an optimal control for linear control system
(\ref{e2-18}). Then there exists a $P_{\ast}\in \mathcal{P}^{\ast}$ such that%
\begin{equation}
(b_{v}(t)p_{t}-f_{v}(t)l_{t})(v-u_{t}^{\ast})\geq0,\text{ }\forall v\in
U,\text{ }a.e.,\text{ }P_{\ast}-a.s., \label{e2-new-18}%
\end{equation}
where $\mathcal{P}^{\ast}$, $b_{v}(t)$ and $f_{v}(t)$ are the same as in
Theorem \ref{th2-3}, $(l$,$p)$ is the solution of the following ordinary
differential equation:%
\[
dl_{t}=[-b_{3}(t)l_{t}+b_{1}(t)p_{t}]dt,\text{ }dp_{t}=[a_{3}(t)l_{t}%
-a_{1}(t)p_{t}]dt,\text{ }l_{0}=1,\text{ }p_{T}=A_{1}l_{T}.
\]

\end{proposition}

\begin{proof}
It is easy to verify that $(l_{t}^{P}$,$p_{t}^{P}$,$q_{t}^{P}$,$N_{t}%
^{P})_{t\leq T}=(l_{t}$,$p_{t}$,$0$,$0)_{t\leq T}$ satisfies the adjoint
equation (\ref{e2-10}) under each $P\in \mathcal{P}$. By Theorem \ref{th2-3},
we get%
\begin{equation}
\inf_{\alpha>0}\sup_{P\in \mathcal{P}^{\ast \alpha}}E_{P}\left[  \int_{0}%
^{T}(b_{v}(t)p_{t}-f_{v}(t)l_{t})(u_{t}-u_{t}^{\ast})dt\right]  \geq0\text{
for any }u\in \mathcal{U}[0,T]. \label{e2-19}%
\end{equation}
Since $\xi(u):=\int_{0}^{T}(b_{v}(t)p_{t}-f_{v}(t)l_{t})(u_{t}-u_{t}^{\ast
})dt\in L_{G}^{T}(\Omega_{1})$, by Lemma 29 in \cite{DHP11}, we obtain that
$E_{P}\left[  \xi(u)\right]  $ is a continuous function under the Prohorov
metric with respect to $P$. Thus we deduce%
\[
\inf_{u\in \mathcal{U}[0,T]}\sup_{P\in \mathcal{P}^{\ast}}E_{P}\left[
\xi(u)\right]  \geq0
\]
by (\ref{e2-19}). It is easy to verify that $E_{P}\left[  \int_{0}^{T}%
(b_{v}(t)p_{t}-f_{v}(t)l_{t})(u_{t}-u_{t}^{\ast})dt\right]  $ is concave in
$P$ and convex in $u$. Therefore, by Sion's minimax theorem, we obtain%
\begin{equation}
\sup_{P\in \mathcal{P}^{\ast}}\inf_{u\in \mathcal{U}[0,T]}E_{P}\left[
\xi(u)\right]  =\inf_{u\in \mathcal{U}[0,T]}\sup_{P\in \mathcal{P}^{\ast}}%
E_{P}\left[  \xi(u)\right]  \geq0. \label{e2-20}%
\end{equation}
Thus we can find a sequence of $\{P_{n}:n\geq1\} \subset \mathcal{P}^{\ast}$
such that $\{P_{n}:n\geq1\}$ converges weakly to $P_{\ast}\in \mathcal{P}%
^{\ast}$ and $\inf_{u\in \mathcal{U}[0,T]}E_{P_{n}}\left[  \xi(u)\right]
\geq-n^{-1}$. Note that $E_{P_{n}}\left[  \xi(u)\right]  \rightarrow
E_{P_{\ast}}\left[  \xi(u)\right]  $ as $n\rightarrow \infty$ for any
$u\in \mathcal{U}[0,T]$. Then we obtain $E_{P_{\ast}}\left[  \xi(u)\right]
\geq0$ for any $u\in \mathcal{U}[0,T]$. It follows from Lusin's theorem that
$M_{G}^{2,\beta}(0,T;U)$ is dense in $M_{P_{\ast}}^{2,\beta}(0,T;U)$ under the
norm $\left(  E_{P_{\ast}}\left[  \left(  \int_{0}^{T}|\eta_{t}|^{2}dt\right)
^{\beta/2}\right]  \right)  ^{1/\beta}$. Thus we obtain
\[
E_{P_{\ast}}\left[  \int_{0}^{T}(b_{v}(t)p_{t}-f_{v}(t)l_{t})(u_{t}%
-u_{t}^{\ast})dt\right]  \geq0\text{ for any }u\in M_{P_{\ast}}^{2,\beta
}(0,T;U),
\]
which implies (\ref{e2-new-18}).
\end{proof}

\section{Sufficient conditions for optimality}

In this section, we still use the notations given in Section 3. Define the
Hamiltonian $H:[0,T]\times \Omega_{T}\times \mathbb{R}^{3}\times \mathbb{R}%
_{+}\times \mathbb{R}^{2+m}\rightarrow \mathbb{R}$ as follows:
\begin{equation}
H(t,x,y,z,l,p,q,v):=b(t,x,y,v)p+\gamma_{t}\sigma(t,x,y,v)q-f(t,x,y,z,v)l.
\label{e3-1}%
\end{equation}
It is easy to check that%
\begin{equation}
H_{v}(t,X_{t}^{\ast},Y_{t}^{\ast},Z_{t}^{\ast},l_{t}^{P},p_{t}^{P},q_{t}%
^{P},u_{t}^{\ast})=b_{v}(t)p_{t}^{P}+\gamma_{t}\sigma_{v}(t)q_{t}^{P}%
-f_{v}(t)l_{t}^{P}. \label{e3-4}%
\end{equation}
The following theorem provides sufficient conditions for optimality.

\begin{theorem}
\label{th3-1}Suppose assumptions (H1)-(H4) hold. Let $u^{\ast}\in
\mathcal{U}[0,T]$ be given and let $(X^{\ast}$,$Y^{\ast}$,$Z^{\ast}$,$K^{\ast
})$ be the solution of (\ref{e1-1}) corresponding to $u^{\ast}$. Additionally,
suppose that $H(\cdot)$ defined in (\ref{e3-1}) is convex with respect to
$(x$,$y$,$z$,$v)$, $\phi$ is convex with respect to $x$ and%
\begin{equation}
\inf_{\alpha>0}\sup_{P\in \mathcal{P}^{\ast \alpha}}E_{P}\left[  \int_{0}%
^{T}(b_{v}(t)p_{t}^{P}+\gamma_{t}\sigma_{v}(t)q_{t}^{P}-f_{v}(t)l_{t}%
^{P})(u_{t}-u_{t}^{\ast})dt\right]  \geq0 \label{e3-2}%
\end{equation}
for any $u\in \mathcal{U}[0,T]$, where $\mathcal{P}^{\ast}$, $b_{v}(t)$,
$\sigma_{v}(t)$, $f_{v}(t)$ and $(l^{P}$,$p^{P}$,$q^{P}$,$N^{P})$ are the same
as in Theorem \ref{th2-3}. Then $u^{\ast}$ is an optimal control.
\end{theorem}

\begin{proof}
For each fixed $u\in \mathcal{U}[0,T]$, set $\tilde{X}_{t}=X_{t}^{u}%
-X_{t}^{\ast}$, $\tilde{Y}_{t}=Y_{t}^{u}-Y_{t}^{\ast}$, $\tilde{Z}_{t}%
=Z_{t}^{u}-Z_{t}^{\ast}$, we have%
\[
\left \{
\begin{array}
[c]{rl}%
d\tilde{X}_{t}= & [b_{x}(t)\tilde{X}_{t}+b_{y}(t)\tilde{Y}_{t}+\tilde
{b}(t)]dt+[\sigma_{x}(t)\tilde{X}_{t}+\sigma_{y}(t)\tilde{Y}_{t}+\tilde
{\sigma}(t)]dB_{t},\\
d\tilde{Y}_{t}= & [f_{x}(t)\tilde{X}_{t}+f_{y}(t)\tilde{Y}_{t}+f_{z}%
(t)\tilde{Z}_{t}+\tilde{f}(t)]dt+\tilde{Z}_{t}dB_{t}+dK_{t}^{u}-dK_{t}^{\ast
},\\
\tilde{X}_{0}= & 0\text{, }\tilde{Y}_{T}=\phi_{x}(X_{T}^{\ast})\tilde{X}%
_{T}+\tilde{\phi}(T),
\end{array}
\right.
\]
where $b_{x}(t)=b_{x}(t,X_{t}^{\ast},Y_{t}^{\ast},u_{t}^{\ast})$, similar
definitions for $b_{y}(t)$, $\sigma_{x}(t)$, $\sigma_{y}(t)$, $f_{x}(t)$,
$f_{y}(t)$ and $f_{z}(t)$, $\tilde{\phi}(T)=-\phi_{x}(X_{T}^{\ast})\tilde
{X}_{T}+\phi(X_{T}^{u})-\phi(X_{T}^{\ast})$,
\[%
\begin{array}
[c]{rl}%
\tilde{b}(t)= & -b_{x}(t)\tilde{X}_{t}-b_{y}(t)\tilde{Y}_{t}+b(t,X_{t}%
^{u},Y_{t}^{u},u_{t})-b(t,X_{t}^{\ast},Y_{t}^{\ast},u^{\ast}),\\
\tilde{\sigma}(t)= & -\sigma_{x}(t)\tilde{X}_{t}-\sigma_{y}(t)\tilde{Y}%
_{t}+\sigma(t,X_{t}^{u},Y_{t}^{u},u_{t})-\sigma(t,X_{t}^{\ast},Y_{t}^{\ast
},u^{\ast}),\\
\tilde{f}(t)= & -f_{x}(t)\tilde{X}_{t}-f_{y}(t)\tilde{Y}_{t}-f_{z}(t)\tilde
{Z}_{t}+f(t,X_{t}^{u},Y_{t}^{u},Z_{t}^{u},u_{t})-f(t,X_{t}^{\ast},Y_{t}^{\ast
},Z_{t}^{\ast},u^{\ast}).
\end{array}
\]
From (\ref{e3-2}), we can find a sequence of $\{P_{n}:n\geq1\} \subset
\mathcal{P}$ such that $\{P_{n}:n\geq1\}$ converges weakly to $P_{\ast}%
\in \mathcal{P}^{\ast}$ and%
\begin{equation}
\lim_{n\rightarrow \infty}E_{P_{n}}\left[  \int_{0}^{T}(b_{v}(t)p_{t}%
^{n}+\gamma_{t}\sigma_{v}(t)q_{t}^{n}-f_{v}(t)l_{t}^{n})(u_{t}-u_{t}^{\ast
})dt\right]  \geq0, \label{e3-5}%
\end{equation}
where $l_{t}^{n}=l_{t}^{P_{n}}$, $p_{t}^{n}=p_{t}^{P_{n}}$ and $q_{t}%
^{n}=q_{t}^{P_{n}}$. For each fixed $n\geq1$, applying It\^{o}'s formula to
$l_{t}^{n}\tilde{Y}_{t}-p_{t}^{n}\tilde{X}_{t}$ under $P_{n}$, we obtain%
\[
\tilde{Y}_{0}=E_{P_{n}}\left[  l_{T}^{n}\tilde{\phi}(T)+\int_{0}^{T}(\tilde
{b}(t)p_{t}^{n}+\gamma_{t}\tilde{\sigma}(t)q_{t}^{n}-\tilde{f}(t)l_{t}%
^{n})dt-\int_{0}^{T}l_{t}^{n}dK_{t}^{u}+\int_{0}^{T}l_{t}^{n}dK_{t}^{\ast
}\right]  .
\]
Since $\phi$ is convex with respect to $x$, we get $\tilde{\phi}(T)\geq0$.
Note that $l_{t}^{n}\geq0$ and $dK_{t}^{u}\leq0$ for $t\leq T$. Then we obtain%
\begin{equation}
\tilde{Y}_{0}\geq E_{P_{n}}\left[  \int_{0}^{T}(\tilde{b}(t)p_{t}^{n}%
+\gamma_{t}\tilde{\sigma}(t)q_{t}^{n}-\tilde{f}(t)l_{t}^{n})dt+\int_{0}%
^{T}l_{t}^{n}dK_{t}^{\ast}\right]  . \label{e3-3}%
\end{equation}
It is easy to verify that%
\[
\tilde{b}(t)p_{t}^{n}+\gamma_{t}\tilde{\sigma}(t)q_{t}^{n}-\tilde{f}%
(t)l_{t}^{n}=H(t,u)-H(t)-H_{x}(t)\tilde{X}_{t}-H_{y}(t)\tilde{Y}_{t}%
-H_{z}(t)\tilde{Z}_{t},
\]
where $H(t,u)=H(t$,$X_{t}^{u}$,$Y_{t}^{u}$,$Z_{t}^{u}$,$l_{t}^{n}$,$p_{t}^{n}%
$,$q_{t}^{n}$,$u_{t})$, $H(t)=H(t$,$X_{t}^{\ast}$,$Y_{t}^{\ast}$,$Z_{t}^{\ast
}$,$l_{t}^{n}$,$p_{t}^{n}$,$q_{t}^{n}$,$u_{t}^{\ast})$, similar definitions
for $H_{x}(t)$, $H_{y}(t)$, $H_{z}(t)$ and $H_{v}(t)$. Since $H(\cdot)$ is
convex with respect to $(x$,$y$,$z$,$v)$, we get%
\[
H(t,u)-H(t)-H_{x}(t)\tilde{X}_{t}-H_{y}(t)\tilde{Y}_{t}-H_{z}(t)\tilde{Z}%
_{t}\geq H_{v}(t)(u_{t}-u_{t}^{\ast}).
\]
Thus we obtain%
\begin{equation}
\tilde{Y}_{0}\geq E_{P_{n}}\left[  \int_{0}^{T}(b_{v}(t)p_{t}^{n}+\gamma
_{t}\sigma_{v}(t)q_{t}^{n}-f_{v}(t)l_{t}^{n})(u_{t}-u_{t}^{\ast})dt+\int
_{0}^{T}l_{t}^{n}dK_{t}^{\ast}\right]  \label{e3-6}%
\end{equation}
by (\ref{e3-4}) and (\ref{e3-3}). Noting that
\[
E_{P_{n}}\left[  \left \vert \int_{0}^{T}l_{t}^{n}dK_{t}^{\ast}\right \vert
\right]  \leq E_{P_{n}}\left[  |K_{T}^{\ast}|\sup_{t\leq T}l_{t}^{n}\right]
\leq \left(  E_{P_{n}}\left[  \sup_{t\leq T}|l_{t}^{n}|^{2}\right]  E_{P_{n}%
}\left[  |K_{T}^{\ast}|^{2}\right]  \right)  ^{1/2}\leq C\left(  E_{P_{n}%
}\left[  |K_{T}^{\ast}|^{2}\right]  \right)  ^{1/2}%
\]
and $E_{P_{n}}\left[  |K_{T}^{\ast}|^{2}\right]  \rightarrow E_{P_{\ast}%
}\left[  |K_{T}^{\ast}|^{2}\right]  =0$ as $n\rightarrow \infty$ by Lemma 29 in
\cite{DHP11}, we deduce%
\begin{equation}
\lim_{n\rightarrow \infty}E_{P_{n}}\left[  \int_{0}^{T}l_{t}^{n}dK_{t}^{\ast
}\right]  =0. \label{e3-7}%
\end{equation}
Thus we obtain $\tilde{Y}_{0}=Y_{0}^{u}-Y_{0}^{\ast}\geq0$ by (\ref{e3-5}),
(\ref{e3-6}) and (\ref{e3-7}). Due to the arbitrariness of $u\in
\mathcal{U}[0,T]$, we obtain that $u^{\ast}$ is an optimal control.
\end{proof}

\begin{corollary}
\label{co3-2}Suppose assumptions (H1)-(H4) hold. Let $u^{\ast}\in
\mathcal{U}[0,T]$ be given and let $(X^{\ast}$,$Y^{\ast}$,$Z^{\ast}$,$K^{\ast
})$ be the solution of (\ref{e1-1}) corresponding to $u^{\ast}$. Additionally,
suppose that $H(\cdot)$ defined in (\ref{e3-1}) is convex with respect to
$(x$,$y$,$z$,$v)$, $\phi$ is convex with respect to $x$ and there exists a
$P_{\ast}\in \mathcal{P}^{\ast}$ such that
\begin{equation}
(b_{v}(t)p_{t}^{P_{\ast}}+\gamma_{t}\sigma_{v}(t)q_{t}^{P_{\ast}}%
-f_{v}(t)l_{t}^{P_{\ast}})(v-u_{t}^{\ast})\geq0,\text{ }a.e.,\text{ }P_{\ast
}-a.s., \label{e3-8}%
\end{equation}
for any $v\in U$, where $\mathcal{P}^{\ast}$, $b_{v}(t)$, $\sigma_{v}(t)$ and
$f_{v}(t)$ are the same as in Theorem \ref{th2-3}, $(l^{P_{\ast}}$%
,$p^{P_{\ast}}$,$q^{P_{\ast}}$,$N^{P_{\ast}})$ is the solution of the adjoint
equation (\ref{e2-10}) under $P_{\ast}$. Then $u^{\ast}$ is an optimal control.
\end{corollary}

\begin{proof}
Since $H_{v}(t,X_{t}^{\ast},Y_{t}^{\ast},Z_{t}^{\ast},l_{t}^{P_{\ast}}%
,p_{t}^{P_{\ast}},q_{t}^{P_{\ast}},u_{t}^{\ast})(v-u_{t}^{\ast})$ is
continuous with respect to $v$, we obtain%
\[
(b_{v}(t)p_{t}^{P_{\ast}}+\gamma_{t}\sigma_{v}(t)q_{t}^{P_{\ast}}%
-f_{v}(t)l_{t}^{P_{\ast}})(v-u_{t}^{\ast})\geq0,\text{ }\forall v\in U,\text{
}a.e.,\text{ }P_{\ast}-a.s.,
\]
by (\ref{e3-8}). Then%
\[
E_{P_{\ast}}\left[  \int_{0}^{T}(b_{v}(t)p_{t}^{P_{\ast}}+\gamma_{t}\sigma
_{v}(t)q_{t}^{P_{\ast}}-f_{v}(t)l_{t}^{P_{\ast}})(u_{t}-u_{t}^{\ast
})dt\right]  \geq0
\]
for any $u\in \mathcal{U}[0,T]$, which implies (\ref{e3-2}). Therefore,
$u^{\ast}$ is an optimal control by Theorem \ref{th3-1}.
\end{proof}

Now we give an example to illustrate how to apply the above theorem to find
the optimal control.

\begin{example}
We consider the following linear control system:%
\[
\left \{
\begin{array}
[c]{rl}%
dX_{t}^{u}= & (-X_{t}^{u}+2aY_{t}^{u}+2u_{t})dt+(X_{t}^{u}+aY_{t}^{u}%
+u_{t})dB_{t},\\
dY_{t}^{u}= & [(a^{3}+a)X_{t}^{u}-Y_{t}^{u}+Z_{t}^{u}+2^{-1}au_{t}%
]dt+Z_{t}^{u}dB_{t}+dK_{t}^{u},\\
X_{0}^{u}= & 1,\text{ }Y_{T}^{u}=aX_{T}^{u},\text{ }t\in \lbrack0,T],
\end{array}
\right.
\]
where the control domain $U=[0,1]$, $T\geq2$ and $a>0$ is an undetermined
constant. Obviously, assumptions (H1), (H2) and (H4) hold. Assumption (H3)
holds by choosing a sufficiently small $a>0$. It is easy to check that
$H(\cdot)$ is convex with respect to $(x$,$y$,$z$,$v)$ and $\phi$ is convex
with respect to $x$. The adjoint equation under $P\in \mathcal{P}$ is%
\[
\left \{
\begin{array}
[c]{l}%
dl_{t}^{P}=(l_{t}^{P}+2ap_{t}^{P}+a\gamma_{t}q_{t}^{P})dt-\gamma_{t}^{-1}%
l_{t}^{P}dB_{t},\\
dp_{t}^{P}=[(a^{3}+a)l_{t}^{P}+p_{t}^{P}-\gamma_{t}q_{t}^{P}]dt+q_{t}%
^{P}dB_{t}+dN_{t}^{P},\\
l_{0}^{P}=1,\text{ }p_{T}^{P}=al_{T}^{p}.
\end{array}
\right.
\]
It is easy to verify that%
\[
p_{t}^{P}=\mu_{t}l_{t}^{P},\text{ }q_{t}^{P}=-\gamma_{t}^{-1}\mu_{t}l_{t}%
^{P},\text{ }N_{t}^{P}=0,\text{ }t\in \lbrack0,T],
\]
where $\mu_{t}=[\exp((1+2a^{2})T)+2a^{2}\exp((1+2a^{2})t)]^{-1}[-a\exp
((1+2a^{2})T)+2a(1+a^{2})\exp((1+2a^{2})t)]$ solves the following ordinary
differential equation:%
\[
\mu_{t}^{\prime}+a\mu_{t}^{2}-\mu_{t}-a^{3}-a=0,\text{ }\mu_{T}=a.
\]
Thus $b_{v}(t)p_{t}^{P}+\gamma_{t}\sigma_{v}(t)q_{t}^{P}-f_{v}(t)l_{t}%
^{P}=(\mu_{t}-2^{-1}a)l_{t}^{P}$. Taking $u_{t}^{\ast}=I_{[0,t_{0}]}(t)$,
where $\mu_{t_{0}}=2^{-1}a$. It is easy to check that%
\[
(\mu_{t}-2^{-1}a)l_{t}^{P}(u_{t}-u_{t}^{\ast})\geq0,\text{ }P-a.s.\text{, for
any }u\in \mathcal{U}[0,T],
\]
which implies (\ref{e3-2}). Thus $u_{t}^{\ast}=I_{[0,t_{0}]}(t)$ is an optimal control.
\end{example}

\end{document}